\documentclass[a4paper,reqno]{amsart}
\usepackage{amsthm}
\usepackage[table]{xcolor}
\usepackage[utf8]{inputenc}
\usepackage[T1]{fontenc}
\usepackage{lipsum}
\usepackage{mathtools}
\usepackage{tabularx}
\usepackage{booktabs}
\usepackage[left=2.0cm, right=2.0cm, top=2.00cm, bottom=2.00cm]{geometry}
\usepackage{amsmath}
\usepackage{helvet}
\usepackage{amsfonts}
\usepackage{amssymb}
\usepackage{graphicx}
\usepackage{mathrsfs} 
\allowdisplaybreaks

\usepackage{cancel}
\usepackage{xcolor}
\definecolor{bluecite}{HTML}{0875b7}
\usepackage[unicode=true,
bookmarksopen={true},
pdffitwindow=true, 
colorlinks=true, 
linkcolor=bluecite, 
citecolor=bluecite, 
urlcolor=bluecite, 
hyperfootnotes=false, 
pdfstartview={FitH},
pdfpagemode= UseNone]{hyperref}
\numberwithin{equation}{section}
\numberwithin{figure}{section}

\newtheorem{proposition}{Proposition}[section]
\newtheorem{theorem}{Theorem}[section]
\newtheorem{lemma}{Lemma}[section]
\newtheorem{corollary}{Corollary}[section]
\newtheorem{definition}{Definition}[section]

\newtheorem{remark}{Remark}[section]

\newcommand{\R}{{\mathbb R}}

\newcommand{\ds}{\displaystyle}
\numberwithin{equation}{section}

\newcommand{\E}{\mathscr E}
\newcommand{\J}{\mathscr{J}}
\newcommand{\Pl}{(\mathscr{P})}

\usepackage{xcolor}
\definecolor{acsblue}{RGB}{17,76,139}
\newcommand{\W}{W_F^{1,p}(M)}
\title{Critical elliptic equations  on non-compact Finsler manifolds}
\author{Csaba Farkas}
\address{Sapientia Hungarian University of Transylvania, Department of Mathematics and Computer Science, Tg. Mure\c s, Romania.}
\email{farkascs@ms.sapientia.ro, farkas.csaba2008@gmail.com}
\newtheorem{innercustomgeneric}{\customgenericname}
\providecommand{\customgenericname}{}
\newcommand{\newcustomtheorem}[2]{%
	\newenvironment{#1}[1]
	{%
		\renewcommand\customgenericname{#2}%
		\renewcommand\theinnercustomgeneric{##1}%
		\innercustomgeneric
	}
	{\endinnercustomgeneric}
}

\newcustomtheorem{customthm}{Theorem}
\usepackage{vwcol} 

\begin{document}
\maketitle

\begin{abstract} 
	In the present paper we deal with a quasilinear elliptic equation involving a critical Sobolev exponent on non-compact Randers spaces. Under very general assumptions on the perturbation, we prove the existence of a non-trivial solution. The approach is based on the direct methods of the calculus of variations. One of the key steps is to prove that the energy functional associated with the problem is weakly lower semicontinuous on small balls of the Sobolev space, which is provided by a general inequality. In the end, we prove Hardy-type inequalities on Finsler manifolds as an application of this inequality.	

\end{abstract}

	%\tableofcontents
		\section{Introduction and main result}
		Since the pioneering work of Brezis and Nirenberg (\cite{BN}) a lot of attention has been paid to  the following problem:
		$$
		\left\{
		\begin{array}{ll}
		-\Delta_{p} u=|u|^{p^*-2}u+\lambda |u|^{p-2}u, & \hbox{ in } \Omega \subset \mathbb{R}^d, \\
		u=0,&  \hbox{ on } \partial\Omega 
		\end{array}
		\right.\eqno{(\mathscr{P})}
		$$
		where $\Omega$ is a bounded domain of $\mathbb{R}^d$, and $\ds p^{*}=\frac{pd}{d-p}$ is the  critical Sobolev exponent, while $\Delta_{p}$ is the $p$-Laplace operator.
		 In \cite{BN} the authors studied the case $p=2$ and they proved that if $\lambda_1(\Omega)$ is the first eigenvalue for $-\Delta$ with Dirichlet boundary conditions, then, if $d\geq 4$ for every $\lambda \in (0,\lambda_1(\Omega))$ there exists a positive solution, if $d=3$ and $\Omega$ is a ball, then, a solution exists if and only if  $\lambda\in \left(\frac{\lambda_1(\Omega)}{4}, \lambda_1(\Omega)\right)$. The proof  is based on a { local Palais Smale} condition and, accordingly, on the construction of minimax levels for the energy functional associated with the problem $\Pl$ in suitable intervals. 
%		Azorero–Peral \cite{Garcia}, Egnell \cite{Egnell}, and Guedda–Véron \cite{GueddaVeron} studied as well problem $\Pl$ with the assumption $p = q$. 	They proved that for every $0<\lambda<\lambda_1(\Omega)$, this problem has a positive solution in high dimensional case, i.e. $d\geq p^2.$ 	
%		
		
	Zou in \cite{Zou} proved that when $d\geq d(p):=[p^2]-[[p^2]-p^2],$ then, problem $\Pl$ has a positive solution if and only if $\lambda\in (0,\lambda_1)$ ($\lambda_1$ being the first eigenvalue for $-\Delta_p$ with Dirichlet boundary conditions). When $d<d(p)$, and $\Omega$ is a ball, then $\Pl$  has no solution if $\lambda\geq \lambda_1$, has a solution for $\lambda\in (\lambda^*,\lambda_1)$, no solution for $\lambda\in(0,\lambda_*)$ (for convenient $0<\lambda_*\leq \lambda^*<\lambda_1)$.
		
		In the literature most of the papers dealing with critical elliptic equations  prove a {local Palais--Smale condition}, which  relies on the  well known  {concentration compactness principle}, see   P.L. Lions in \cite{Lions}, which is one of the most powerful tools in the case of lack of compactness. In this regard we mention the result of Chabrowski (\cite{Chabrowski}) where the author applies  concentration compactness principle to a non-homogeneous problem with non-constant coefficients. 
		
		Recently in \cite{FF_Collectanea} the authors proposed an alternative method to the problem $$
		\left\{
		\begin{array}{ll}
		-\Delta_p u=
		|u|^{p^*-2}u+ g(u), & \hbox{ in } \Omega \\
		u=0, & \hbox{on } \partial \Omega
		\end{array}
		\right.\eqno{ \Pl}
		$$
		where $\Omega$ is a bounded domain,  by employing the  direct methods of calculus of variations. Indeed,  they proved that the energy functional $\E$ associated with the problem  is locally  sequentially weakly lower semicontinuous (LSC), and with direct, simple arguments they prove that $\E$ has a local minimum, which is a weak solution of the problem $\Pl$. Later the idea of this paper was successfully applied in different context as well, see \cite{Bisci_London,Bisci_Jde}.  Here we would like to highlight the paper by M. Squassina  \cite{Squassina}, where the lower semicontinuity argument together with a Mountain Pass geometry has been exploited  for a multiplicity result when $g(u)=g(x,u)=\lambda|u|^{q-2}u+\varepsilon h(x)$, $p<q<p^*$,  $\lambda$ is large and $\varepsilon$ small. Note that the proof of LSC in our case is based on a general inequality (see Lemma \ref{fontosLemma}) that differs from the argument of Squassina. 
		
				In the case of non-compact Riemannian manifolds the study of critical elliptic equations changes dramatically, see for instance Yau \cite{Yau}. Thus as one expects, the curvature plays a crucial role in the study of critical elliptic problem, in this connection we mention the work of Aviles and McOwen \cite{Aviles}. 
		
		In the light of the above, the aim of the present manuscript is to answer the question of whether an existence result for a critical elliptic equation can be given on non-compact Finsler manifolds. Therefore we consider the following problem:
		$$
		\left\{
		\begin{array}{ll}
		-\Delta_{p,F} u+|u|^{p-2}u=
		\mu |u|^{p^*-2}u+\lambda \alpha(x)h(u), & \hbox{ in } M, \\
		u\in W^{1,p}_F(M)
		\end{array}
		\right.\eqno{(\mathscr{P}_{\lambda, \mu})}
		$$
		where $M$ is a $d$-dimensional non-compact Randers space endowed with the Finsler metric defined as 
		\begin{equation} \label{Randers_metric}
		F(x,y) = \sqrt{g_x(y,y)} + \beta_x(y), \quad (x,y) \in TM,
		\end{equation} 
		where $g$ is a Riemannian metric and $\beta_x$ is a $1$-form on $M$.  $\Delta_{p,F}$ is the \textit{Finsler $p$-Laplace operator},  $h:\R\to\R$ is a continuous function, and $\ds p^{*}=\frac{pd}{d-p}$ is the  critical Sobolev exponent.

		In the case of non-compact Finsler manifolds the study of critical elliptic equation is quite delicate. First,  it turns out that if $M$ is a non-compact Finsler-Hadamard manifold with infinite reversibility constant, then the Sobolev space $W_F^{1,2}(M)$ is not necessary a vector space, see Farkas, Krist\'aly, Varga \cite{FarkasKristalyVarga_Calc}. On the other hand by Farkas, Krist\'aly, Mester \cite{FKM_Morrey} the continuous Sobolev-type embedding do not necessarily hold. 
		
		Beside the pathological phenomena appearing in the  theory of Sobolev spaces on Finsler manifolds, one of the main obstacles in dealing with existence and multiplicity results for quasilinear problems with critical nonlinearity is represented by the weakly lower semicontinuity of the energy functional associated with the problem $(\mathscr{P}_{\lambda,\mu})$. Thus 
		$${E}_{\lambda,\mu}(u)=\frac{1}{p}\int\limits_{M} \left({F^*}^p(x,Du)+|u|^p\right)\,\mathrm{d}V_F(x)-\frac{\mu}{p^{*}}\int\limits_{M} |u|^{p^{*}}\,\mathrm{d}V_F(x)-\lambda \int\limits_{M}\alpha(x)H(u)\,\mathrm{d}V_F(x),$$
		is not sequentially weakly lower semicontinuous in $\W$ and does not satisfy the well known  Palais Smale condition. However, $E_{\lambda,\mu}$ is of class $C^1$ in $\W$ and its critical points turn out to be the weak solutions of problem $(\mathscr{P}_{\lambda,\mu})$.
		
		In the sequel we state our main result, in order to do that, we need some notations and assumptions.  Let $(M,F)$ be a $d$-dimensional Randers space endowed with the Finsler metric \eqref{Randers_metric}, such that $(M,g)$ is a complete non-compact Riemannian manifold  of bounded geometry, i.e. with Ricci curvature bounded from below, i.e.  ${\rm 
			Ric}_{(M,g)}\geq k (d-1),$ for some $k\in \mathbb{R}$, having positive injectivity radius.   Denote by  $\label{betanorma}\|\beta\|_g(x):=\sqrt{g_x^*(\beta_x,\beta_x)}$  for every $x\in M,$ where $g^*$ is the co-metric of $g$, and $\displaystyle a=\sup_{x\in M}\|\beta\|_{g(x)}$. 
		
		Let $h:\mathbb{R} \to \mathbb{R}$ be a continuous function. 
		\begin{description}
			\item[\textbf{($\mathbf{H}$)}] For each $s\in \mathbb{R}$, let $\displaystyle H(s)=\intop_0^s h(t)\,\mathrm{d}t$, and we assume that:\begin{itemize}
				\item[($h_1$)] there exist  $1<r<p\leq q<p^*$ and $c_1,c_2>0$  such 
				that $$\ds |h(t)|\leq c_1|t|^{q-1}+c_2|t|^{r-1},\ \forall t\in \mathbb{R};$$
				\item[($h_2$)] \label{h2}$\ds \liminf_{s\to0}\frac{H(s)}{|s|^p}=+\infty$.
			\end{itemize}
		\end{description}
%		On the function $\alpha$ we assume that:
%		\begin{description}
%			\item[\textbf{($\mathbf{\alpha}$)}]\label{alpha} $\alpha:M \to \mathbb{R}$ non-negative continuous function which depends on $d_F(x_0,\cdot)$, for $x_0\in M$, i.e. $\alpha(x)=\alpha_{0}(d_F(x_0,x))$ and assume that $\alpha_0$ is a non-increasing function and  $$\alpha_{0}(s)\sinh\left(|k| \frac{s}{1-a}\right)^{d-1}\sim \frac{1}{s^\gamma},$$ for some $\gamma>1$, whenever $s\to \infty$.
%		\end{description} 
		
		\newpage
		
		Our main result reads as follows:
		\begin{theorem}\label{main}Let $(M,F)$ be a $d$-dimensional Randers space endowed with the Finsler metric \eqref{Randers_metric} such that $\ds a:=\sup_{x\in M}\|\beta\|_g(x)<1$  and $(M,g)$ is a complete non-compact Riemannian manifold of bounded geometry, i.e.  ${\rm 
				Ric}_{(M,g)}\geq k (d-1)$  with $k< 0$, having positive injectivity radius. Let $G$ be a coercive, compact  connected subgroup of $\mathrm{Isom}_F(M)$ such that $\mathrm{Fix}_M(G)=\{x_0\}$ for some $x_0\in M$. Let $h:\mathbb{R} \to \mathbb{R}$ be a continuous function verifying (\textbf{H}), and  let $\alpha:M \to \mathbb{R}$ be $\alpha:M \to \mathbb{R}$ non-negative continuous function which depends on $d_F(x_0,\cdot)$, for $x_0\in M$, i.e. $\alpha(x)=\alpha_{0}(d_F(x_0,x))$ and assume that $\alpha_0$ is a non-increasing function and  $$\alpha_{0}(s)\sinh\left(|k| \frac{s}{1-a}\right)^{d-1}\sim \frac{1}{s^\gamma},$$ for some $\gamma>1$, whenever $s\to \infty$..  Then for every $\mu>0$, there exists $\lambda_*>0$ such that for every $\lambda<\lambda_*$ the problem ($\mathscr{P}_{\lambda,\mu}$) has a non-zero $G$-invariant weak solution. 
			
		\end{theorem}
	The organization of the paper is the following.  In Section \ref{sec:background} we present some preliminary results and notions from Riemann and Finsler geometry, while  Section \ref{Proofs} is devoted to the proof of Theorem \ref{main}.  At the end, in Section  \ref{finalremarks}, we shall present our result also on Finsler--Hadamard manifolds..
		\section{Mathematical background}\label{sec:background}
		\subsection{Elements from Riemannian geometry}
		Let $(M,g)$ be a homogeneous complete non-compact Riemannian manifold with $\mathrm{dim}M=d$.  
		Let $T_xM$ be the tangent space at $x \in M$, $\displaystyle TM =
		\bigcup_{x\in M}T_xM$ be the tangent bundle, and $d_g : M \times M
		\to [0, +\infty)$ be the distance function associated to the
		Riemannian metric $g$. Let $B_g(x, \rho) = \{y \in M : d_g(x, y) <
		\rho \}$ be the open metric ball with center $x$ and radius $\rho
		> 0$. If $\mathrm{d}v_g$ is the canonical volume element on $(M, g)$, the
		volume of a bounded open set $\Omega \subset M$ is 
		$\mathrm{Vol}_g(\Omega) = \displaystyle\int_{\Omega} \, \mathrm{d}v_g= \mathcal H^d(\Omega)$.
		If ${\text d}\sigma_g$ denotes
		the $(d-1)-$dimensional Riemannian measure induced on $\partial \Omega$
		by $g$, then
		$$\mathrm{Area}_g(\partial \Omega)=\displaystyle\int_{\partial \Omega} {\text d}\sigma_g=\mathcal H^{d-1}(\partial \Omega)$$ stands for the area of $\partial \Omega$ with
		respect to the metric $g$.
		
		For every $p>1$, the norm of $L^p(M)$ is given by
		$$\|u\|_{L^p(M)}=\left(\displaystyle\int_M |u|^p\,\mathrm{d}v_g\right)^{1/p}.$$ Let $u:M\to \mathbb R$ be a function of
		class $C^1.$ If $(x^i)$ denotes the local coordinate system on a
		coordinate neighbourhood of $x\in M$, and the local components of
		the differential of $u$ are denoted by 
		$u_i=\frac{\partial	u}{\partial x_i}$, then the local components of the gradient  $\nabla_g u$ are
		$u^i=g^{ij}u_j$. Here, $g^{ij}$ are the local components of
		$g^{-1}=(g_{ij})^{-1}$. In particular, for every $x_0\in M$ one has
		the eikonal equation
		\begin{equation}\label{dist-gradient}
		|\nabla_g d_g(x_0,\cdot)|=1\ {\rm a.e. \ on}\ M.
		\end{equation}
		%The Laplace-Beltrami operator is given by 
		%$\Delta_g u={\rm div}(\nabla_g u)$, whose expression in a local chart of associated
		%coordinates $(x^i)$ is 
		%$$\Delta_g u=g^{ij}\left(\frac{\partial^2 u}{\partial x_i\partial x_j} 
		%- \Gamma_{ij}^k\frac{\partial u}{\partial x_k}\right),$$ 
		%where $\Gamma_{ij}^k$ are the coefficients of the
		%Levi-Civita connection.
		
		When no confusion arises, if $X,Y\in T_x M$, we simply write $|X|$ and
		$\langle X,Y\rangle$ instead of the norm $|X|_x$ and inner product
		$g_x(X,Y)=\langle X,Y\rangle_x$, respectively. 
		
		The $L^p(M)$ norm of $\nabla_g u: M \to TM$ is given by
		$$\|\nabla_g u\|_{L^p(M)}=\left(\displaystyle\int_M |\nabla_gu|^p{\rm d}v_g\right)^\frac{1}{p}.$$  
		The space $W^{1,p}_g(M)$ is the completion of $C_0^\infty(M)$ with respect to the norm
		$$\|u\|^p_{W^{1,p}_g(M)}={\|u\|_{L^p(M)}^p+\|\nabla_g u\|_{L^p(M)}^p}.$$

		\subsection{Elements from Finsler geometry}
		Let $M$ be a connected $n$-dimensional $C^{\infty}$ manifold and
		$TM=\bigcup_{x\in M}T_{x}M$ is its tangent bundle, the pair $(M,F)$
		is a Finsler manifold if the continuous function $F:TM\to[0,\infty)$
		satisfies the conditions:
		\begin{itemize}
			\item $F\in C^{\infty}(TM\setminus\{0\});$
			\item $F(x,ty)=tF(x,y)$ for all $t\geq0$ and $(x,y)\in TM;$
			\item $g_{ij}(x,y):=\left[\frac{1}{2}F^{2}(x,y)\right]_{y^{i}y^{j}}$ is
			positive definite for all $(x,y)\in TM\setminus\{0\}.$
		\end{itemize}
		If $F(x,ty)=|t|F(x,y)$ for all $t\in\mathbb{R}$ and $(x,y)\in TM,$
		we say that the Finsler manifold $(M,F)$ is reversible. 	Clearly, the Randers
		metric $F$ is symmetric, i.e.
		$F(x,-y)=F(x,y)$ for every $(x,y)\in TM,$ if and only if $\beta=0$ (which means that $(M,F)=(M,g)$ is the original Riemannian manifold).

		Let $\pi ^{*}TM$ be the pull-back bundle of the tangent bundle $TM$
		generated by the natural projection $\pi:TM\setminus\{ 0 \}\to M,$ see Bao,
		Chern and Shen \cite[p. 28]{BaoChernShen}. The vectors of the pull-back bundle $\pi
		^{*}TM$ are denoted by $(v;w)$ with $(x,y)=v\in TM\setminus\{ 0 \}$ and $%
		w\in T_xM.$ For simplicity, let $\partial_i|_v=(v;\partial/\partial x^i|_x)$
		be the natural local basis for $\pi ^{*}TM$, where $v\in T_xM.$ One can
		introduce on $\pi ^{*}TM$ the \textit{fundamental tensor} $\mathfrak{g}$
		%and {\it Cartan tensor} $A$
		by
		\begin{equation}  \label{funda-Cartan-tensors}
		\mathfrak{g}_v:=\mathfrak{g}(\partial_i|_v,\partial_j|_v)=\mathfrak{g}_{ij}(x,y)
		\end{equation}
		%respectively,
		where $v=y^i{(\partial}/{\partial x^i})|_x.$ 
		
		Let $u,v\in T_xM$ be two non-collinear vectors and $\mathcal{S}=\mathrm{span}%
		\{u,v\}\subset T_xM$. By means of the curvature tensor $R$, the \textit{flag
			curvature} of the flag $\{\mathcal{S},v\}$ is defined by
		\begin{equation}  \label{ref-flag}
		\mathbf{K}(\mathcal{S};v) =%
		\frac{\mathfrak{g}_v(R(U,V)V, U)}{\mathfrak{g}_v(V,V) \mathfrak{g}_v(U,U) - \mathfrak{g}_v(U,V)^{2}},
		\end{equation}
		where $U=(v;u),V=(v;v)\in \pi^*TM.$ If $(M,F)$ is Riemannian, the flag
		curvature reduces to the well known sectional curvature. If $\mathbf{K}(%
		\mathcal{S};v)\leq 0$ for every choice of $U$ and $V$, we say that $(M,F)$
		has \textit{non-positive flag curvature}, and we denote by $\mathbf{K}\leq 0$%
		. $(M,F)$ is a \textit{Finsler-Hadamard manifold} if it is simply connected,
		forward complete with $\mathbf{K}\leq 0$.
		
		Let $\{e_i\}_{i=1,...,n}$ be a basis for $T_xM$ and $\mathfrak{g}_{{ij}}^v=\mathfrak{g}_v(e_i,e_j)$%
		. The \textit{mean distortion} $\zeta:TM\setminus \{0\}\to (0,\infty)$ is
		defined by $\zeta(v)=\frac{\sqrt{\mathrm{det}(g_{ij}^v)}}{\sigma_F}$. The
		\textit{mean covariation} $\mathbf{S}:TM\setminus\{0\}\to \mathbb{R}$ is
		defined by
		\begin{equation*}
		\mathbf{S}(x,v)=\frac{\mathrm{d}}{\mathrm{d}t}(\ln \zeta(\dot\sigma_v(t)))\big|%
		_{t=0},
		\end{equation*}
		where $\sigma_v$ is the geodesic such that $\sigma_v(0)=x$ and $\dot
		\sigma_v(0)=v.$ %The pointwise norm function of mean covariation
		%${\bf S}$ is defined by $$\|{\bf S}\|_x=\sup_{v\in T_xM\setminus
		%\{0\}}\frac{{\bf S}(x,v)}{F(x,v)}.$$ Finally, $\|{\bf
		%S}\|=\sup_{x\in M}\|{\bf S}\|_x.$
		We say that $(M,F)$ has \textit{vanishing mean covariation} if $\mathbf{S}%
		(x,v)= 0$ for every $(x,v)\in TM$, and we denote by $\mathbf{S}=0$.

	Let $\sigma: [0,r]\to M$ be a piecewise $C^{\infty}$ curve. The
	value $ L_F(\sigma)= \displaystyle\int_{0}^{r} F(\sigma(t),
	\dot\sigma(t))\,{\text d}t $ denotes the \textit{integral length} of
	$\sigma.$  For $x_1,x_2\in M$, denote by $\Lambda(x_1,x_2)$ the set
	of all piecewise $C^{\infty}$ curves $\sigma:[0,r]\to M$ such that
	$\sigma(0)=x_1$ and $\sigma(r)=x_2$. Define the {\it distance
		function} $d_{F}: M\times M \to[0,\infty)$ by
	\begin{equation}\label{quasi-metric}
	d_{F}(x_1,x_2) = \inf_{\sigma\in\Lambda(x_1,x_2)}
	L_F(\sigma).
	\end{equation}
	One clearly has that $d_{F}(x_1,x_2) =0$ if and only if $x_1=x_2,$
	and that $d_F$ verifies the triangle inequality.
	
	The {\it Hausdorff  volume form} ${\text d}V_F$ on the Randers space $(M,F)$ is given by
	\begin{equation}\label{randers-volume-def}
	{\text d}V_F(x)=\left(1-\|\beta\|^2_g(x)\right)^\frac{d+1}{2}\mathrm{d}v_g,
	\end{equation}
	where $dv_g$ denotes the canonical Riemannian volume form induced by $g$ on $M$.
	
	For every $(x,\alpha)\in T^*M$, the  {\it polar transform} (or,
	co-metric) of $F$ from (\ref{Randers_metric}) is
	\begin{equation}\label{polar-transform}
	F^*(x,\alpha) = \sup_{y\in T_xM\setminus
		\{0\}}\frac{\alpha(y)}{F(x,y)}=\frac{\sqrt{g_x^{*2}(\alpha,\beta)+(1-\|\beta\|_g^2(x))\|\alpha\|_g^2(x)}-g_x^{*}(\alpha,\beta)}{1-\|\beta\|_g^2(x)}.
	\end{equation}
	
	Let $u:M\to \mathbb{R}$ be a differentiable function in the distributional sense. The \textit{gradient} of $u$ is defined by
	\begin{equation}  \label{grad-deriv}
	\boldsymbol{\nabla}_F u(x)=J^*(x,Du(x)),
	\end{equation}
	where $Du(x)\in T_x^*M$ denotes the (distributional) \textit{derivative} of $u$ at $x\in M$ and $J^*$ is the Legendre transform  given by $$J^*(x,y):=\frac{\partial}{\partial y}\left(\frac{1}{2}F^{*2}(x,y)\right).$$  In local coordinates, one has
	\begin{equation}  \label{derivalt-local}
	Du(x)=\sum_{i=1}^n \frac{\partial u}{\partial x^i}(x)\mathrm{d}x^i,  \ \boldsymbol{\nabla}_F u(x)=\sum_{i,j=1}^n h_{ij}^*(x,Du(x))\frac{\partial u}{\partial x^i}(x)\frac{\partial}{\partial x^j}.
	\end{equation}

	In general, note that $u\mapsto\boldsymbol{\nabla}_F u $ is not linear. If $x_0\in M$ is
	fixed, then due to \cite{Otha-Sturm}, for a.e. $x\in M$ one has 
	\begin{equation}  \label{tavolsag-derivalt}
	F^*(x,D d_F(x_0,x))=F(x,\boldsymbol{\nabla}_F d_F(x_0,x))=D d_F(x_0,x)(%
	\boldsymbol{\nabla}_F d_F(x_0,x))=1
	\end{equation}
	Let $p>1.$ The norm of $L^p(M)$ is given by
	$$\|u\|_{p}=\left(\displaystyle\int_M |u|^p\,\mathrm{d}V_F(x)\right)^{1/p}.$$
	Let $X$ be a vector field on $M$. In a local coordinate system $(x^i)$  the \textit{divergence} is defined by div$(X)=\frac{1}{\sigma_F}\frac{\partial}{\partial x^i}(\sigma_F X^i),$ where 
	$$\sigma_F(x)=\frac{\omega_{d}}{\mathrm{Vol}(\{y=(y^i):\ F(x,y^i\frac{\partial }{\partial x^i})<1\})}.$$ 
	
	%The \textit{Finsler-Laplace operator} is defined by
	%\begin{equation*}
	%\boldsymbol{\Delta}_F u=\mathrm{div}(\boldsymbol{\nabla}_F u),
	%\end{equation*}

	The Finsler $p$-Laplace operator is defined by $$\boldsymbol{\Delta}_{F,p}u=\mathrm{div}({F^*}^{p-2}(Du) \cdot \boldsymbol{\nabla}_F u),$$
	while the Green theorem reads as: for every $v\in C_0^\infty(M)$,
	\begin{equation}  \label{Green}
	\int_M v\boldsymbol{\Delta}_{F,p} u \,{\mathrm d}V_F(x)=-\int_M
	{F^*}^{p-2}(Du) Dv(\boldsymbol{\nabla}_F u)\,{\text d}V_F(x),
	\end{equation}
	see   \cite{Otha-Sturm} and 
	\cite{Shen_Konyv} for $p=2$. Note that in general 
	$\boldsymbol{\Delta}_{F,p} (-u) \neq -\boldsymbol{\Delta}_{F,p} u.$
	When $(M,F)=(M,g)$, the Finsler-Laplace operator is the usual Laplace-Beltrami operator, 
	%	while for a Minkowski space 
	%	$(\mathbb{R}^n,F)$,  
	%	$\boldsymbol{\Delta}_Fu=\mathrm{div} (F^*(Du)\nabla F^*(Du))=\mathrm{div} (F(\nabla u)\nabla F(\nabla u))$ 
	%	is precisely the Finsler-Laplace operator considered
	%	by \cite{Cianchi-Salani}.
	
	We introduce the function space
	associated with $\left(M,F\right)$, namely let 
	$$W^{1,p}_F(M)=\left\{u\in W^{1,p}_\mathrm{loc}(M): \int_M {F^*}^p(x,Du(x))\mathrm{d}V_F(x)<+\infty \right\}$$ 
	be the closure of $C^\infty(M)$ with respect to the (asymmetric) norm 
	$$\|u\|_{F}=\left(\int_M {F^*}^p(x,Du(x))\,\mathrm{d}V_F(x)+\int_M |u(x)|^p\,\mathrm{d}V_F(x) \right)^{\frac{1}{p}}.$$ 
	
	Note that, when $M$ is a Randers-space endowed with the Finsler metric defenied in \eqref{Randers_metric}, such that $\sup_{x\in M}\|\beta\|_{g}(x)<1$, then the function space $W^{1,p}_F(M)$ is a Sobolev space, see \cite{FarkasKristalyVarga_Calc,KristalyRudas_NATMA}.

	We notice that  the \textit{reversibility constant} associated with $F$ (see (\ref{Randers_metric}) is given by 
	\begin{equation}  \label{reverzibilis}
	r_{F}=\sup_{x\in M}r_F(x)\ \ \ \mathrm{where}\ \ \ r_F(x)=\sup_{\substack{ y \in T_x M\setminus \{0\}}} \frac{F(x,y)}{F(x,-y)}=\frac{1+\|\beta\|_g(x)}{1-\|\beta\|_g(x)},
	\end{equation}
	see \cite{Rademacher, Yuan-Zhao}.
	%Note that $r_{F}\geq 1$ (possibly, $r_{F}=+\infty$), and $r_{F}= 1$ if
	%and only if $(M,F)$ is Riemannian. 
	%In the same way, we define the constant $r_{F^*}$ associated with $F^*$ and one has $r_{F^*}=r_F.$
	In a similar way, one can define the  \textit{uniformity constant} of $F$ by:
	\begin{equation} \label{uniformity_const}
	l_{F}=\inf_{x\in M}l_F(x)\ \ \ \mathrm{where}\ \ \ l_F(x)= \inf_{y,v,w\in
		T_xM\setminus \{0\}}\frac{g_{(x,v)}(y,y)}{g_{(x,w)}(y,y)}=\left(\frac{1-\|\beta\|_g(x)}{1+\|\beta\|_g(x)}\right)^2,
	\end{equation}
	see \cite{Egloff,FarkasKristalyVarga_Calc,Otha-Sturm}.  The definition of $l_{F}$ in turn shows that
	\begin{equation}  \label{eq:2uni}
	F^{*2}( x,{t\xi+(1-t)\beta} ) \le tF^{*2}(x,\xi) +(1-t)F^{*2}(x,\beta)
	-{l_{F}}t(1-t) F^{*2}(x,\beta-\xi)
	\end{equation}
	for all $x\in M$, $\xi,\beta \in T_x^*M$ and $t\in [0,1]$.

	\subsection{Sobolev embeddings on Randers spaces} 	 As we already mentiond in our case the function space $W^{1,p}_F(M)$ is a Sobolev space. In this section  we focus on the (compact) embedding $W^{1,p}_F(M)\hookrightarrow L^q(M)$. 
	
	% Even in $\mathbb{R}^d$ it is well known that  the  compactness of the embedding $W^{1,p}(\mathbb{R}^d)\hookrightarrow L^q(\mathbb{R}^d)$ need not 
	% hold, see \cite{Adams}. On the other hand, it was proved by Berestycki–Lions (see \cite{BL} and \cite{Lions}, and also \cite{Strauss,Willem}) that if $\ p\leq d$ then the embedding $W^{1,p}_{\mathrm{rad}}(\mathbb{R}^d)\hookrightarrow L^q(\mathbb{R}^d)$ is compact  whenever $p<q<p^*$, where $W^{1,p}_{\mathrm{rad}}(\mathbb{R}^d)$ stands for the subspace of radially symmetric functions of $W^{1,p}(\mathbb{R}^d)$. 
	
	% The Berestycki–Lions-type theorem has  been  established on Riemannian manifolds by  \cite{HebeyVaugon}, see also \cite[Theorems 9.5 \& 9.6]{Hebey-konyv1}. More precisely, if $G$ is a compact subgroup of the group of global isometries of the complete Riemannian manifold $(M, g)$ then (under some additional assumptions on the geometry of $(M, g)$ and some assumptions on the orbits under the action of $G$),
	% the embedding $W^{1,p}_G(M)\hookrightarrow L^q(M)$ turns out to be compact, where $W^{1,p}_G(M)$ denotes the set of $G$-invariant functions of the Sobolev space $W^{1,p}(M)$.  Such compactness results have been extended to non-compact metric measure spaces as well, see \cite{Gorka}, and generalized to Lebesgue–Sobolev spaces $W_G^{1,p(\cdot)}(M)$ in the setting of complete Riemannian manifolds, see \cite{LengyelekJFA}.   
	
	Recently, in \cite{SkrzypczakTintarev} the authors proved that, if $(M,g)$ is a $d$-dimensional homogeneous Hadamard manifold and $G$ is a compact connected subgroup of the group of global isometries of $(M, g)$ such that $\mathrm{Fix}_M(G)$ is a singleton, then the subspace of the $G$-invariant functions of $W_g^{1,p}(M)$ is compactly embedded into $L^q(M)$ whenever $p\leq d$ and $p<q<p^*$.  
	For further use we use the following definition
	\begin{definition}
		We say that a continuous action of a group $G$ on a complete Riemannian manifold $M$ is coercive if for every $t > 0$, the set $$O_t=\{x\in M: \mathrm{diam}Gx<t\}$$ is bounded.
	\end{definition}
	
	It was also pointed out that, see \cite{SkrzypczakTintarev,TintarevKonyvs},  the condition of a single fixed point is restrictive, because on manifolds where Sobolev embeddings hold, one has a necessary and sufficient condition in terms of coercivity condition on the group $G$.  According to this, for proving a compact emnedding result for Sobolev spaces defined on Randers spaces, we use the following result, see \cite{skrzypczak2020compact,TintarevKonyvs}:
	\begin{customthm}{A}[Theorem 7.10.12, \cite{TintarevKonyvs}]\label{RiemannBeagyazas}Let $G$ be a compact, connected group of isometries of a $d$-dimensional non-compact connected Riemannian manifold $M$ of bounded geometry. Let $1<p<d$ and $p<q<p^*$. Then the subspace $W_G^{1,p}(M)$ is compactly embedded into $L^q(M)$ if and only if $G$ is coercive.	
	\end{customthm}
	\begin{remark}
		Based on \cite{FKM_Morrey}, one could expect an alternative proof of the aforementioned embedding. Indeed, if we assume that there exists $\kappa=\kappa(G,d) > 0$  such that for every $y\in M$ with $d_g(x_0,y) \geq 1$, one has
		$$\mathcal H^{l}(\mathcal{O}_G^y) \geq \kappa \cdot d_g(x_0,y),$$ 
		where $l=l(y)=\dim \mathcal{O}_G^y \geq 1$, where $\mathcal{O}_G^y=\{\tau(x):\tau \in G\}$, then $W_G^{1,p}(M)$ is compactly embedded into $L^q(M)$. For the proof, see \cite[Theorem 4.1]{FKM_Morrey}.
	\end{remark}
	
	Denote by $W^{1,p}_{F,G}(M)$ the subspace of $G$-invariant functions of $W^{1,p}_{F}(M)$, where $G$ is a subgroup of $\mathrm{Isom}_F(M)$, such that $G$ is coercive. In this case we have the following result:
	\begin{customthm}{B}[Theorem 1.3.,  \cite{FKM_Morrey}]\label{beagyazas}
		Let $(M,F)$ be a $d$-dimensional Randers space endowed with the Finsler metric \eqref{Randers_metric}, such that $(M,g)$ is a $d$-dimensional non-compact connected Riemannian manifold $M$ of bounded geometry, and let  $G$ be a coercive compact connected subgroup of $\mathrm{Isom}_F(M)$.
		If $1<p<d$, $q\in (p,p^*)$ and $\displaystyle \sup_{x\in M}\|\beta\|_g(x)<1$, then the embedding $W^{1,p}_{F}(M)\hookrightarrow L^q(M)$ is continuous, while the embedding  $W^{1,p}_{F,G}(M)\hookrightarrow L^q(M)$ is compact. 
	\end{customthm}
	First, recall that $0<a:=\sup_{x\in M}\|\beta\|_g(x)<1$, then the proof of the previous theorem is based on the following two inequalities
	\begin{equation}\label{randers-volume}
		(1-a^2)^\frac{d+1}{2}{\text d}v_g\leq {\text d}V_F(x)\leq {\text d}v_g,
	\end{equation}
and \begin{equation}\label{becslesek}
	\frac{\|\xi\|_g(x)}{1+a}\leq F^*(x,\xi)\leq \frac{\|\xi\|_g(x)}{1-a}.
\end{equation}

		\section{Proof of the main result}\label{Proofs}
		In this section we prove Theorem \ref{main}. First we deal with the lower semicontinuity of the energy functional $${E}_{\lambda,\mu}(u)=\mathscr{F}_\mu(u)-\lambda \mathscr{K}(u),$$ where $$\mathscr{F}_\mu(u)=\frac{1}{p}\int_{M}\left(F^{*^{p}}(x,Du(x))+|u|^{p}\right)\,\mathrm{d}V_{F}(x)-\frac{\mu}{p^{*}}\int_{M}|u|^{p^{*}}\,\mathrm{d}V_{F}(x),$$ and $$\mathscr{K}(u)=\int_{M}\alpha(x)H(u(x))\,\mathrm{d}V_{F}(x).$$ First we prove the Finslerian version of the following inequality (see Lemma \ref{fontosLemma}): if $p\geq2$ then for $a,b\in\mathbb{R}^{N}$, we have that  (see \cite[Lemma 4.2]{Lindqvist}):
		\begin{equation}
		|b|^{p}\geq|a|^{p}+p\langle|a|^{p-2}a,b-a\rangle+2^{1-p}|a-b|^{p}\label{eq:Lindqvist}
		\end{equation}
		Note that, Lemma \ref{fontosLemma} is indispensable for the lower semicontinuity of $\mathscr{F}_\mu$, see Proposition \ref{lsc_kritikus}. 
		
		\begin{lemma}\label{fontosLemma}
			Let $(M,F)$ be a Finsler manifold, then we have the following inequality:
			\[
			p(\beta-\xi)\left(F^{*p-2}(x,\xi)J^{*}(x,\xi)\right)+\frac{l_{F}^{\frac{p}{2}}}{2^{p-1}}F^{*^{p}}(x,\beta-\xi)+F^{*^{p}}(x,\xi)\leq F^{*p}(x,\beta),\ \forall\xi,\beta\in T_{x}^{*}M.
			\]
		\end{lemma}
		
		\begin{proof}
			From (\ref{eq:2uni}) with the choice $\ds t=\frac{1}{2},$ one
			has that 
			\begin{equation}
			F^{*2}\left(x,\frac{\xi+\beta}{2}\right)\le\frac{F^{*2}(x,\xi)+F^{*2}(x,\beta)}{2}-\frac{l_{F}}{4}F^{*2}(x,\beta-\xi).\label{eq:l_F_2}
			\end{equation}
			Now, it is clear that if $p\geq2,$ then $p\mapsto\left(a^{p}+b^{p}\right)^{\frac{1}{p}}$
			is non-increasing, moreover by Hölder inequality , one has that 
			\[
			\left(a^{p}+b^{p}\right)^{\frac{1}{p}}\leq2^{\frac{1}{p}-\frac{1}{2}}\left(a^{2}+b^{2}\right)^{\frac{1}{2}}\leq\left(a^{2}+b^{2}\right)^{\frac{1}{2}}.
			\]
			Thus, applying this inequality, one has that 
			\begin{align*}
			F^{*p}\left(x,\frac{\xi+\beta}{2}\right)+\frac{l_{F}^{\frac{p}{2}}}{2^{p}}F^{*^{p}}(x,\beta-\xi) & \leq2^{\frac{1}{p}-\frac{1}{2}}\cdot\left(F^{*2}\left(x,\frac{\xi+\beta}{2}\right)+\frac{l_{F}}{4}F^{*2}(x,\beta-\xi)\right)^{\frac{p}{2}} \\&\overset{\eqref{eq:l_F_2}}{\leq}\left(\frac{F^{*2}(x,\xi)+F^{*2}(x,\beta)}{2}\right)^{\frac{p}{2}}.
			\end{align*}
			On the other hand, by convexity we have that 
			\[
			\left(\frac{F^{*2}(x,\xi)+F^{*2}(x,\beta)}{2}\right)^{\frac{p}{2}}\leq\frac{1}{2}F^{*p}(x,\xi)+\frac{1}{2}F^{*p}(x,\beta).
			\]
			
			On the other hand, by convexity 
			\[
			F^{*^{p}}\left(x,\frac{\xi+\beta}{2}\right)\geq F^{*p}(x,\xi)+\frac{p}{2}(\beta-\xi)\left(F^{*p-2}(x,\xi)J^{*}(x,\xi)\right),
			\]
			thus 
			\[
			F^{*p}\left(x,\frac{\xi+\beta}{2}\right)+\frac{l_{F}^{\frac{p}{2}}}{2^{p}}F^{*^{p}}(x,\beta-\xi)\geq F^{*p}(x,\xi)+\frac{p}{2}(\beta-\xi)\left(F^{*p-2}(x,\xi)J^{*}(x,\xi)\right)+\frac{l_{F}^{\frac{p}{2}}}{2^{p}}F^{*^{p}}(x,\beta-\xi).
			\]
			
			Putting toghether, one has that 
			\[
			\frac{p}{2}(\beta-\xi)\left(F^{*p-2}(x,\xi)J^{*}(x,\xi)\right)+\frac{l_{F}^{\frac{p}{2}}}{2^{p}}F^{*^{p}}(x,\beta-\xi)+\frac{1}{2}F^{*^{p}}(x,\xi)\leq\frac{1}{2}F^{*^{p}}(x,\beta),
			\]
			or for every $\xi,\beta\in T_{x}^{*}M$ the following inequality hold true
			\begin{equation}
			p(\beta-\xi)\left(F^{*p-2}(x,\xi)J^{*}(x,\xi)\right)+\frac{l_{F}^{\frac{p}{2}}}{2^{p-1}}F^{*^{p}}(x,\beta-\xi)+F^{*^{p}}(x,\xi)\leq F^{*p}(x,\beta),\label{eq:ezkell}
			\end{equation}
			which proves the Lemma.
		\end{proof}
		
		Inequality \eqref{eq:ezkell} has been proved by Xia on compact Finsler manifold, see \cite[Lemma 3.1]{Xia}

		\begin{remark}
			In the case of Minkowski spaces the inequality (\ref{eq:ezkell})
			reads as
			\[
			F^{*p}(\beta)\geq F^{*p}(\xi)+pF^{*p-1}(\xi)\langle\nabla F^{*}(\xi),\beta-\xi\rangle+\frac{l_{F}^{\frac{p}{2}}}{2^{p-1}}F^{*p}(\beta-\xi),
			\]
			
			where 
			\[
			l_{F}=\min\left\{ \left\langle \nabla^2\left(\frac{F^{2}}{2}\right)(x)y,y\right\rangle :\ F(x)=F(y)=1\right\} .
			\]
			
		\end{remark}
		
		Furthermore we denote by $\kappa_{p^*}$ the inverse  of the Sobolev embedding constant of $W^{1,p}_F(M)\hookrightarrow L^{p^*}(M)$, that is, $$(\kappa_{p^*})^{-1}=\inf_{u\neq 0}\frac{\|u\|_F}{\|u\|_{p^*}}.$$

		Now, we are in the position to prove the following proposition:
			\begin{proposition}\label{lsc_kritikus}
		For every $\mu>0$ and for every $0<\rho<\rho^*\equiv\ds\left(\frac{1}{\mu}\frac{p^*}{p}\frac{l_F^{\frac{p}{2}}}{2^{p-1}\kappa_{p^*}^{p^*}}\right)^\frac{1}{p^*-p}$, the restriction of  $\mathscr{F}_\mu$ to $B_\rho:=\{u\in W^{1,p}_F(M):\ \|u\|_F\leq \rho \}$ is sequentially weakly lower semicontinuous.
			
%			\begin{enumerate}
%				\item[(i)] 
%				\item[(ii)] for every $\rho>0$ and for every $0<\mu<\mu^*\equiv\ds\frac{1}{\rho^{p^*-p}}\frac{p^*}{p\kappa_{p^*}^{p^*}} \frac{l_F^{\frac{p}{2}}}{2^{p-1}}$, the restriction of  $\mathscr{F}_\mu$ to $B_\rho(u)$ is sequentially weakly lower semicontinuous.
%			\end{enumerate}
		\end{proposition}
		\begin{proof}
			
			Let $u_{n}\rightharpoonup u$ in $W_{F}^{1,p}(M),$ thus we have
			that\\
			${\displaystyle \frac{1}{p}\int_{M}F^{*{p}}(x,Du_{n}(x))\,\mathrm{d}V_{F}(x)-\frac{1}{p}\int_{M}F^{*{p}}(x,Du(x))\,\mathrm{d}V_{F}(x)\geq}$
			\begin{align*}
			& \geq p\int_{M}(Du_{n}(x)-Du(x))\left(F^{*p-2}(x,Du(x))J^{*}(x,Du(x))\right)\,\mathrm{d}V_{F}(x)\\&+\frac{l_{F}^{\frac{p}{2}}}{2^{p-1}}\int_{M}F^{*^{p}}(x,Du_{n}(x)-Du(x))\,\mathrm{d}V_{F}(x).
			\end{align*}
			On the other hand, by (\ref{eq:Lindqvist})
			\begin{align*}
			\int_{M}|u_{n}|^{p}\,\mathrm{d}V_{F}(x)-\int_{M}|u|^{p}\,\mathrm{d}V_{F}(x) & \geq p\int_{M}|u|^{p-2}u(u_{n}-u)\,\mathrm{d}V_{F}(x)+\frac{1}{2^{p-1}}\int_{M}|u_{n}-u|^{p}\,\mathrm{d}V_{F}(x)\\
			& \geq p\int_{M}|u|^{p-2}u(u_{n}-u)\,\mathrm{d}V_{F}(x)+\frac{l_{F}^{\frac{p}{2}}}{2^{p-1}}\int_{M}|u_{n}-u|^{p}\,\mathrm{d}V_{F}(x).
			\end{align*}
			Summing up
			\begin{align*}
			\|u_{n}\|_F^{p}-\|u\|_F^{p}  \geq & p\int_{M}(Du_{n}(x)-Du(x))\left(F^{*p-2}(x,Du(x))J^{*}(x,Du(x))\right)\,\mathrm{d}V_{F}(x)+\\
			& +p\int_{M}|u|^{p-2}u(u_{n}-u)\,\mathrm{d}V_{F}(x)+\frac{l_{F}^{\frac{p}{2}}}{2^{p-1}}\|u_{n}-u\|_F^{p}.
			\end{align*}
			By the well know Brezis-Lieb Lemma one has 
			\[
			\liminf_{n\to\infty}\left(\int_{M}|u_{n}|^{p^{\ast}}\,\mathrm{d}V_{F}(x)-\int_{M}|u|^{p^{\ast}}\,\mathrm{d}V_{F}(x)\right)=\liminf_{n\to\infty}\int_{M}|u_{n}-u|^{p^{\ast}}\,\mathrm{d}V_{F}(x).
			\]
			Putting all together, we have that 
			\begin{align*}
			\liminf_{n\to\infty}\left(\mathscr{F}_\mu(u_{n})-\mathscr{F}_\mu(u)\right)&\geq\liminf_{n\to\infty}\left(\frac{1}{p}\cdot\frac{l_{F}^{\frac{p}{2}}}{2^{p-1}}\|u_{n}-u\|_F^{p}-\frac{\mu}{p^{*}}\int_{M}|u_{n}-u|^{p^{*}}\,\mathrm{d}V_{F}(x)\right)\\ &\geq\liminf_{n\to\infty}\|u_{n}-u\|_F^{p}\left(\frac{1}{p}\cdot\frac{l_{F}^{\frac{p}{2}}}{2^{p-1}}-\frac{\mu}{p^{*}}\kappa_{p^*}^{p^*}\rho^{p^{*}-p}\right)\geq 0,
			\end{align*}
			which concludes the proof.
		\end{proof}	
		 It is worth to mention that this result generalizes  the lower semicontinuity of functionals involving critical Sobolev exponent in both Euclidean and Riemannian cases.

From the assumption (\hyperref[alpha]{$\alpha$}) on can easily observe that $\alpha \in L^\infty(M)$. On the other hand, by the layer cake representation it follows that 
\begin{align*}
\int_{M} \alpha(x)\, \mathrm{d}V_F(x) &= \int_M \alpha_{0}(d_F(x_0,x))\,\mathrm{d}V_F(x)\\
&=\int_0^\infty \mathrm{Vol}_F\left(\left\{x\in M:\, \alpha_{0}(d_F(x_0,x))>t \right\}\right)\,\mathrm{d}t \  \ \  [\mbox{chang. of var. } t=\alpha_0(z)]\\ &= \int_\infty^0 \mathrm{Vol}_F(B_F(x_0,z)) \alpha_{0}'(z)\,\mathrm{d}z=\int_{0}^\infty \mathrm{Vol}_F(B_F(x_0,z)) (-\alpha_{0}(z))\,\mathrm{d}z.
\end{align*}

Since, $(M,F)$ is a Randers space, and $a:=\sup_{x}\Vert \beta \Vert _{g}(x)<1$ one can observe that $$(1-a)d_g(x,y)\leq d_F(x,y)\leq (1+a)d_g(x,y),$$ thus \begin{equation}
\label{gombokegymasban}
B_g\left(x_0,\frac{z}{1+a}\right)\subseteq B_F(x_0,z)\subseteq B_g\left(x_0,\frac{z}{1-a}\right).\end{equation} On account of this inclusion and the Bishop-Gromov inequality (see \cite{Chavel}) and  from the assumption (\hyperref[alpha]{$\alpha$}),  one can see that \begin{align*}
\int_{M} \alpha(x)\, \mathrm{d}V_F(x) =\int_0^\infty \mathrm{Vol}_F(B_F(x_0,z)) (-\alpha_{0}'(z))\,\mathrm{d}z &\leq \int_{0}^{\infty} \mathrm{Vol}_g\left(B_g\left(x_0,\frac{z}{1-a}\right)\right)(-\alpha_0'(z))\,\mathrm{d}z \\ &= \int_{0}^\infty \mathrm{Area}_g \left(B_g\left(x_0,\frac{z}{1-a}\right)\right) \alpha_{0}(z)\,\mathrm{d}z\\ &\leq c \int_{0}^\infty \sinh\left(k\frac{z}{1-a}\right)^{d-1}\alpha_{0}(z)\,\mathrm{d}z<+\infty,
\end{align*}
thus $\alpha\in L^1(M)$.

 We have that the energy is $G$-invariant, more precisely (see  \cite[Lemma 5.1]{FKM_Morrey}):

	\begin{lemma}\label{iso-lemma}
	Let $G$ be a compact connected coercive subgroup of $\mathrm{Isom}_F(M)$ with $\mathrm{Fix}_M(G)=\{x_0\}$ for some $x_0\in M$. Then ${E}_{\lambda,\mu}$ is $G$-invariant, i.e., for every $\tau \in G$ and $u\in W^{1,p}_{F}(M)$ one has ${E}_{\lambda,\mu}(\tau
	 u)={E}_{\lambda,\mu}(u)$. 	
\end{lemma}

For further use,  we restrict the energy functional to the space $W^{1,p}_{F,G}(M)$. For simplicity, in the following we denote $$\mathcal{E}_{\lambda,\mu} = {E}_{\lambda,\mu} \rvert_{W^{1,p}_{F,G}(M)}, \  \mbox{ and } \mathscr{K}_G=\mathscr{K}\rvert_{W^{1,p}_{F,G}(M)}.$$

The principle of symmetric criticality of Palais (see Krist\'aly, R\u adulescu and Varga \cite[Theorem
1.50]{KVR-book}) and Lemma \ref{iso-lemma} imply that the critical points of
$\mathcal{E}_{\lambda,\mu }= {E}_{\lambda,\mu} \rvert_{W^{1,p}_{F,G}(M)}$ are also critical points of the original functional ${E}_\lambda$.

		\begin{lemma}\label{lsc2}
			The functional $\mathscr{K}_G:W^{1,p}_{F,G}(M)\to \mathbb{R}$ is weakly lower semicontinuous on $W^{1,p}_{F,G}(M)$.
		\end{lemma}
		\begin{proof}
			Consider $\{u_n\}$ a sequence in $W^{1,p}_{F,G}(M)$ which converges weakly to $u\in W^{1,p}_{F,G}(M)$, and suppose that $$\mathscr{K}_G(u_n)\cancel{\to} \mathscr{K}_G(u_n),\ \mbox{ as }n\to \infty.$$ Thus, there exist $\varepsilon>0$ and a subsequence of $\{u_n\}$, denoted again by $\{u_n\}$, such that  $$0<\varepsilon\leq |\mathscr{K}_G(u_n)-\mathscr{K}_G(u)|,\ \mbox{ for every }n\in \mathbb{N}.$$ Thus, by the mean value theorem, there exists $\theta_n\in(0,1)$ such that
			\begin{align}
			\varepsilon &\leq |\mathscr{K}_G(u_n)-\mathscr{K}_G(u)| \leq \int_{M} \alpha(x)|H(u_n(x))-H(u(x))|\,\mathrm{d}V_F(x)\nonumber \\
			&\leq \int_M \alpha(x)|h(u+\theta_n(u_n-u))|\cdot |u_n-u|\,\mathrm{d}V_F(x)\nonumber \\&\leq \int_{M} \alpha(x) |u_n-u| \left(c_1|u+\theta_n(u_n-u)|^{r-1}+c_2|u+\theta_n(u_n-u)|^{q-1}\right)\,\mathrm{d}V_F(x)\nonumber \\ &\leq \int_M \alpha(x)\left(c_1|u_n|^{r-1}|u_n-u|+c_2|u_n|^{q-1}|u_n-u|\right)\,\mathrm{d}V_F(x). \label{ezthasznalomkesobb}
			\end{align}
			For further use, let $m,b>0$ be two real numbers, such that 
			$$\begin{cases}
				\ds p<\frac{m}{m-r+1}<p^*,\ \mbox{ or } \ \frac{p^*(r-1)}{p^*-1}<m<\frac{p(r-1)}{p-1},\\
				\ds p<\frac{b}{b-q+1}<p^*,\  \mbox{ or }\  \frac{p^*(q-1)}{p^*-1}<b<\frac{p(q-1)}{p-1}.
			\end{cases}$$
			In this case, by H\"older inequality we have that 
			$$c_1\int_{M}\alpha(x)|u_n|^{r-1}|u_n-u|\,\mathrm{d}V_F(x)\leq c_1\|\alpha\|_{\infty} \|u_n-u\|_{\frac{m}{m-r+1}}\|u_n\|_m^{r-1},$$ and $$c_2\int_{M}\alpha(x)|u_n|^{q-1}|u_n-u|\,\mathrm{d}V_F(x)\leq c_2\|\alpha\|_{\infty} \|u_n-u\|_{\frac{b}{b-q+1}}\|u_n\|_b^{q-1}.$$ 
			Combining the above two inequalities with Theorem \ref{beagyazas}, we get that $$\int_{M}\alpha(x)|H(u_n(x))-H(u(x))|\,\mathrm{d}V_F(x)\to 0,$$ which is a contradiction in the light of \eqref{ezthasznalomkesobb}, which concludes the proof of the lemma.
		\end{proof}
			According to Lemma \ref{lsc_kritikus} and Lemma \ref{lsc2} the functional $\mathcal{E}_{\lambda,\mu}$ is weakly lower semicontinuous on small ball of the Sobolev space $W^{1,p}_{F,G}(M)$. 
			Now, we are in the position to prove our main result.
			\begin{proof}[Proof of Theorem \ref{main}] For any $r>0$ denote $B_r=\{u\in W^{1,p}_{F,G}: \|u\|_F\leq r\},$ 
			and define  $\J_{\mu,\lambda}:W^{1,p}_{F,G} \to \R,$ by $${\J_{\mu,\lambda}}(u)=\frac{\mu}{p^{*}}\int\limits_{M} |u|^{p^{*}}\,{\text d}V_F(x)+\lambda \int\limits_{M}\alpha(x)H(u)\,{\text d}V_F(x).$$
			
			For any $\lambda, \mu, \rho>0$ define
			\begin{equation}\label{VP}
			\varphi_{\mu,\lambda}(\rho):=\inf_{\|u\|<\rho}
			\frac{\sup_{ B_\rho}\J_{\mu,\lambda}-\J_{\mu,\lambda}(u)}{\rho^p-\|u\|_F^p} \ \mbox{ and } \psi_{\mu,\lambda}(\rho):=\sup_{B_\rho}\J_{\mu,\lambda}.
			\end{equation}
%			and
%			$$$$
			We claim that, under our assumptions, there exist $\lambda, \mu, \rho>0$ such that
			\begin{equation}\label{min}
			\varphi_{\mu,\lambda}(\rho)<\frac{1}{p}.
			\end{equation}
			It is easy to observe that, in order to grantee, the previous inequality it is enough to prove that there exist $\lambda, \mu, \rho>0$ such that
			\begin{equation}\label{amiveguliskell}\inf_{\sigma<\rho}\frac{\psi_{\mu,\lambda}(\rho)-\psi_{\mu,\lambda}(\sigma)}{\rho^p-\sigma^p}<\frac{1}{p}.\end{equation}
			
			Thus, in the sequel we focus on the  inequality \eqref{amiveguliskell}. Now, if $\sigma=\rho-\varepsilon$, for some $\varepsilon>0$, then we get 
			
			\begin{eqnarray*}
				\frac{\psi_{\mu,\lambda}(\rho)-\psi_{\mu,\lambda}(\sigma)}{\rho^p-\sigma^p}&=&\frac{\psi_{\mu,\lambda}(\rho)-\psi_{\mu,\lambda}(\rho-\varepsilon)}
				{\rho^p-(\rho-\varepsilon)^p}=\frac{\psi_{\mu,\lambda}(\rho)-\psi_{\mu,\lambda}(\rho-\varepsilon)}{\varepsilon}
				\cdot\frac{-\frac{\varepsilon}{\rho}}{\rho^{p-1}\left[(1-\frac{\varepsilon}{\rho})^p-1\right]},
			\end{eqnarray*}
			thus (\ref{amiveguliskell}) is  fulfilled if there exist $\mu,\rho>0$ such that
			\begin{equation}\label{min2}
			\limsup_{\varepsilon\to 0^+}\frac{\psi_\mu(\rho)-\psi_\mu(\rho-\varepsilon)}
			{\varepsilon}<\rho^{p-1}.
			\end{equation}
			We are going to estimate the right-hand side of \eqref{min2}, by assumption (\textbf{H}) we have that 
			
			\begin{eqnarray*}
				\frac{\psi_{\mu,\lambda}(\rho)-\psi_{\mu,\lambda}(\rho-\varepsilon)}
				{\varepsilon}&\leq&
				\frac{1}{\varepsilon}
				\sup_{\|v\|_F\leq 1}\int\limits_{M}\left|\,\int\limits_{(\rho-\varepsilon) v(x)}^{\rho v(x)}\mu|t|^{p^*-1}+\lambda \alpha(x)|h(t)|dt\right|\,\mathrm{d}V_F(x)\\&\leq&
				\frac{\kappa_{p^*}^{p^*}\mu}{p^*}\left|\frac{\rho^{p^*}\!-\!(\rho-\varepsilon)^{p^*}}{\varepsilon}\right|
				+\lambda c_1\frac{\kappa_{p^*}^q\|\alpha \|_{\frac{p^*}{p^*-q}}}{q}
				\cdot \left|\frac{\rho^{q}\!-\!(\rho-\varepsilon)^{q}}{\varepsilon}\right|\\&+&\lambda c_2
				\frac{\kappa_{p^*}^r\|\alpha\|_{\frac{p^*}{p^*-r}}}{r}
				\left|\frac{\rho^{r}\!-\!(\rho-\varepsilon)^{r}}{\varepsilon}\right|.
			\end{eqnarray*}
		Calculating the limit, as $\varepsilon\to 0$, we have that 
			\begin{equation}\label{psisup}\limsup_{\varepsilon\rightarrow 0}\frac{\psi_\mu(\rho)-\psi_\mu(\rho-\varepsilon)}
				{\varepsilon}\leq
				\mu\kappa_{p^*}^{p^*}\rho^{p^*-1}+\lambda c_1\kappa_{p^*}^q\|\alpha \|_{\frac{p^*}{p^*-q}}
				\rho^{q-1}+\lambda c_2\kappa_{p^*}^r\|\alpha\|_{\frac{p^*}{p^*-r}}\rho^{r-1}.
			\end{equation}
			Let $\rho_0>0$ such that $$\frac{\rho_0^{p-1}-\mu \kappa_{p^*}^{p^*}\rho_0^{p^*-1}}{c_1\kappa_{p^*}^q\|\alpha \|_{\frac{p^*}{p^*-q}}
				\rho_0^{q-1}+ c_2\kappa_{p^*}^r\|\alpha\|_{\frac{p^*}{p^*-r}}\rho_0^{r-1}}=\max_{t>0}\frac{t^{p-1}-\mu \kappa_{p^*}^{p^*}t^{p^*-1}}{c_1\kappa_{p^*}^q\|\alpha \|_{\frac{p^*}{p^*-q}}
				t^{q-1}+ c_2\kappa_{p^*}^r\|\alpha\|_{\frac{p^*}{p^*-r}}t^{r-1}},$$ and consider $\rho_\mu:=\min\{\rho_0,\rho^*\}$. Thus if $$0<\lambda<\lambda_*:=\frac{\rho_\mu^{p-1}-\mu \kappa_{p^*}^{p^*}\rho_\mu^{p^*-1}}{c_1\kappa_{p^*}^q\|\alpha \|_{\frac{p^*}{p^*-q}}
				\rho_\mu^{q-1}+ c_2\kappa_{p^*}^r\|\alpha\|_{\frac{p^*}{p^*-r}}\rho_\mu^{r-1}},$$ therefore, based on \eqref{psisup} there exists $(\lambda,\rho)$ such that $$\limsup_{\varepsilon\rightarrow 0}\frac{\psi_\mu(\rho)-\psi_\mu(\rho-\varepsilon)}
			{\varepsilon}<\rho_\mu^{p-1},$$ so that (\ref{min}) is fulfilled.
			
			Condition (\ref{min}) implies the existence of $u_0\in W^{1,p}_{F,G}$ with $\|u_0\|_F<\rho_\mu$ such that
			\begin{equation}\label{u0}
			\E_\mu(u_0)<\frac{1}{p}\rho_\mu^p- \J_\mu(u)
			\end{equation}
			for every $u\in B_{\rho_\mu}$. Since the energy $\E_{\lambda,\mu}$ is sequentially weakly lower semicontinuous in $B_{\rho_\mu}$, its restriction to the ball has a global minimum $u_\ast$. If $\|u_\ast\|=\rho_\mu$, then, from ($\ref{u0}$)
			$\E_{\lambda,\mu}(u_\ast)=\frac{1}{p}\rho_\mu^p- \J_\mu(u_\ast)>\E_{\lambda,\mu}(u_0),$
			a contradiction. It follows that $u_\ast$ is a local minimum for $\E_{\lambda,\mu}$ with $\|u_\ast\|_F<\rho_\mu$, hence in particular,  a  weak solution of problem $(\mathscr P_{\lambda, \mu})$. 
			
			In the sequel we prove that $u_*$ is not identically zero.  In order to prove that our solution is non-trivial, we show that there exists a function for which the energy is negative. To this end, observe that from the assumption on the function $\alpha$, it is clear that, there exists  $R>0$ such that $\alpha_R:=\underset{d_{F}(x_{0},x)\leq R}{\mathrm{essinf}}\alpha(x) > 0.$ Let $\zeta>0$ such that $\zeta<\displaystyle R\frac{1-a}{1+a}$. Now we can define the following function: 
			$$u_{R,\zeta}=\begin{cases}
			0, & x\in M\setminus B_F(x_0,R)\\
			\frac{1}{R-\zeta}(R-d_F(x_0,x)),&  x\in B_F(x_0,R)\setminus B_F(x_0,\zeta)\\
			1,&  x\in B_F(x_0,\zeta)
			\end{cases}$$
			
			Due to (\hyperref[h2]{$h_2$}), it is clear that there exists a sequence $\theta_j$, with $\theta_j\to 0$ as $j\to +\infty$, such that \begin{equation}\label{fontosH}
		c\theta_j^p\leq H(\theta_j), \ \forall c>0,
			\end{equation} and $j$ large enough.  
			
			Consider the following function $u_1:=\theta_j u_{R,\zeta}$. In the sequel we are going to estimate $\mathcal{E}_{\mu,\lambda}(u_1)$.
			
			First of all, recall that $r_{F}>0$ is the reversibility constant on $(M,F)$, see \eqref{reverzibilis}, i.e.  by the eikonal identity \eqref{tavolsag-derivalt} we have that $\ds\frac{1}{r_F}\leq F^*(x,-Dd_F(x_0,x))\leq r_F.$ Therefore, 
		\begin{align*}
		\int_M {F^*}^p(x,Du_1(x))\,\mathrm{d}V_F(x) &= \int_M {F^*}^p(x,\theta_jDu_{R,\zeta}(x))\,\mathrm{d}V_F(x) \\&=\theta_j^p\left(\frac{1}{R-\zeta}\right)^p \int\limits_{B_F(x_0,R)\setminus B_F(x_0,\zeta)} {F^*}^p(x,-Dd_F(x_0,x)(x))\,\mathrm{d}V_F(x) \\ &\leq \theta_j^p\left(\frac{1}{R-\zeta}\right)^p {r_F^p} \mathrm{Vol}_F(B_F(x_0,R)),
		\end{align*}
		and 
		\begin{align*}
		\int_M |u_1|^{p}\,\mathrm{d}V_F(x)&=\theta_j^p \int_M |u_{R,\zeta}|^p\,\mathrm{d}V_F(x)\leq \theta_j^p \mathrm{Vol}_F(B_F(x_0,R)).
		\end{align*}
		On the other hand 
			\begin{align*}
		\int_M |u_1|^{p^*}\,\mathrm{d}V_F(x)&=\theta_j^p \int_M |u_{R,\zeta}|^{p^*}\,\mathrm{d}V_F(x)\theta_j^{p^*} \mathrm{Vol}_F(B_F(x_0,\zeta)).
		\end{align*}
		Finally, by \eqref{fontosH} we have that 
		\begin{align*}
		\int_M \alpha(x)H(u_1)\,\mathrm{d}V_F(x)&\geq c\theta_j^p\int_{B_F(x_0,\zeta)}\alpha(x)\,\mathrm{d}V_F(x)\geq c\theta_j^p\alpha_R \mathrm{Vol}_F(B_F(x_0,\zeta)).
		\end{align*}
		Putting all together, one has that 
		\begin{align*}
		\mathcal{E}_{\mu,\lambda}(u_1)\leq& \left(\frac{1}{p}\left(\frac{1}{R-\zeta}\right)^p {r_F^p} \mathrm{Vol}_F(B_F(x_0,R)+\frac{1}{p}\mathrm{Vol}_F(B_F(x_0,R)))-c\alpha_R \mathrm{Vol}_F(B_F(x_0,\zeta))\right)\theta_j^p\\&- \frac{\theta_j^{p^*}}{p^{*}}\mathrm{Vol}_F(B_F(x_0,\zeta)).
		\end{align*}
		By choosing $c>0$ large enough, one can easily see that $0$ is not a local minimizer of the energy functional, and hence $u_*\neq 0$. Which concludes the proof.
		\end{proof}
	We end this section with the following corollary:
	\begin{corollary}
		Let $(M,g)$ be a $d$-dimensional complete non-compact Riemannian manifold of bounded geometry, i.e.  ${\rm Ric}_{(M,g)}\geq k (d-1)$ with $k<0$, having positive injectivity radius. Let $G$ be a coercive, compact  connected subgroup of $\mathrm{Isom}_G(M)$ such that $\mathrm{Fix}_M(G)=\{x_0\}$ for some $x_0\in M$. Let $h:\mathbb{R} \to \mathbb{R}$ be a continuous function verifying (\textbf{H}), and  let $\alpha:M \to \mathbb{R}$ be a continuous function radially symmetric function w.t.r $x_0\in M$, i.e. there exists $\alpha_{0}:[0,+\infty)\to \mathbb{R}$ such that $\alpha(x)=\alpha_{0}(d_g(x_0,x)),\ \forall x\in M$, which satisfies $$\alpha_{0}(s)\sinh\left(|k| \frac{s}{1-a}\right)^{d-1}\sim \frac{1}{s^\gamma},$$ for some $\gamma>1$, whenever $s\to \infty$. Then for every $\mu>0$, there exists $\lambda_*>0$ such that for every $\lambda<\lambda_*$ the problem $$
		\left\{
		\begin{array}{ll}
			-\Delta_{p,g} u+|u|^{p-2}u=
			\mu |u|^{p^*-2}u+\lambda \alpha(x)h(u), & \hbox{ in } M, \\
			u\in W^{1,p}_g(M)
		\end{array}
		\right.
		$$ has a non-zero $G$-invariant weak solution. 
	\end{corollary}

	\section{Critical elliptic equations on Hadamard type Randers spaces}\label{finalremarks} 
	Through this section, assume that  $(M,F)$ is a $d$-dimensional Randers space endowed with the Finsler metric \eqref{Randers_metric} such that $\sup_{x\in M}\|\beta\|_g(x)<1$ and $(M,g)$ is a Hadamard manifold.    
	In this case a similar result as Theorem \ref{main} can be obtained, which reads as follows:
	\begin{theorem}\label{HadLetezes}
		 Let $(M,F)$ be a $d$-dimensional Randers space endowed with the Finsler metric \eqref{Randers_metric} such that $\sup_{x\in M}\|\beta\|_g(x)<1$ and $(M,g)$ is a Hadamard manifold.   Let $G$ be a compact  connected subgroup of $\mathrm{Isom}_F(M)$ such that $\mathrm{Fix}_M(G)=\{x_0\}$ for some $x_0\in M$. Let $h:\mathbb{R} \to \mathbb{R}$ be a continuous function verifying (\textbf{H}), and let $\alpha\in L^1(M)\cap L^\infty(M)\setminus\{0\}$ be a non-negative and radially symmetric function w.t.r $x_0\in M$, i.e. there exists $\alpha_{0}:[0,+\infty)\to \mathbb{R}$ such that $\alpha(x)=\alpha_{0}(d_F(x_0,x)),\ \forall x\in M$. 	Then for every $\mu>0$, there exists $\lambda_*>0$ such that for every $\lambda<\lambda_*$ the problem $(\mathscr{P}_{\lambda,\mu})$ has a non-zero $G$-invariant weak solution. 
	\end{theorem}

	 Note that, such a result is an extension of \cite{Bisci_Jde}. 	Moreover, if with sectional curvature of $(M,g)$ bounded above by $-\mathtt{k}^2$, $\mathtt{k}>0$, then under the same assumptions of Theorem \ref{HadLetezes} for every $\mu>0$, there exists $\lambda_*>0$ such that for every $\lambda<\lambda_*$ the  problem 	$$
	\left\{
	\begin{array}{ll}
	-\Delta_{p,F} u=
	\mu |u|^{p^*-2}u+\lambda \alpha(x)h(u), & \hbox{ in } M, \\
	u\in W^{1,p}_F(M)
	\end{array}
	\right.
	$$ has a non-zero $G$-invariant weak solution. 
	
	The proof of the above statement is similar to the proof of Theorem \ref{main}. The key point is a McKean-type inequality (see for instance \cite[Theorem 0.6]{Yin-He}):
	$$\lambda_{1,g} :=\inf_{u \in W^{1,p}_g(M)}\frac{\displaystyle \int_M|\nabla_{g} u|^p\,\mathrm{d}v_g}{\displaystyle \int_M |u|^p\,\mathrm{d}v_g}\geq \left(\frac{(d-1)\mathtt{k}}{p}\right)^p,$$ 
	therefore, 
	$$\int_M |\nabla_g u|^p\,\mathrm{d}v_g\geq \frac{(d-1)^p\mathtt{k}^p}{p^p+(d-1)^p\mathtt{k}^p}\|u\|^p_{W^{1,p}_g(M)},\ \ \ u \in W^{1,p}_g(M).$$ 
	Using \eqref{randers-volume}, \eqref{becslesek},  and denoting by $\mathtt{C}_{\mathtt{k},p}:=\frac{(1-a^2)^{(d+1)/2}}{(1+a)^p} \frac{(d-1)^p\mathtt{k}^p}{p^p+(d-1)^p\mathtt{k}^p}$, we obtain that for every $u$
	\begin{align*} \label{D-norma_osszehasonlitas}
	\left(\frac{1}{1-a}\right)^p\|u\|^p_{W^{1,p}_g(M)}&\geq\int_M {F^*}^p(x,Du(x))\,\mathrm{d}V_F(x) \geq  \mathtt{C}_{\mathtt{k},p} \|u\|^p_{W^{1,p}_g(M)}.
	\end{align*}

	\section*{Acknowledgment}
	The author was supported by  UEFISCDI/CNCS grant PN-III-P4-ID-PCE2020-1001. 
	
	The author would like to thank the anonymous referee who provided useful and detailed comments on a previous/earlier version of the manuscript.
	
%	{ \footnotesize\bibliographystyle{apa}
%		\bibliography{hivatkozasok}}
%		\bibliographystyle{Falpha}
%		\bibliography{hivatkozasok}

\end{document}